\documentclass[journal,12pt,onecolumn,draftclsnofoot,]{IEEEtran}
\usepackage{graphicx} 
\usepackage{amsmath}
\usepackage{outlines}
\usepackage{amssymb}
\usepackage{natbib}
\usepackage{mathtools}
\usepackage[symbol]{footmisc}    
\usepackage{amsthm}   
\usepackage{verbatim} 
\usepackage{subfigure}  
\usepackage[dvipsnames]{xcolor}
\usepackage{hyperref}
\usepackage{enumitem}
\usepackage{tkz-euclide}
\usetikzlibrary{decorations.pathmorphing}
\usetikzlibrary {arrows.meta}

\renewcommand{\thesubfigure}{Fig. \thefigure.\arabic{subfigure}}
\makeatletter
\renewcommand{\p@subfigure}{}
\renewcommand{\@thesubfigure}{\thesubfigure\hskip\subfiglabelskip}
\makeatother

\hypersetup{
    colorlinks=true,
    linkcolor=blue,
    filecolor=magenta,      
    urlcolor=blue,
    citecolor={blue}
}
\newtheorem{theorem}{Theorem}[section]
\newtheorem{obs}{Observation}[section]
\newtheorem{lemma}[theorem]{Lemma}
\newtheorem{defi}{Definition}[section]

\newtheorem{scheme}[theorem]{Broadcast Scheme}

\DeclarePairedDelimiter{\ceil}{\lceil}{\rceil}      
\DeclarePairedDelimiter{\floor}{\lfloor}{\rfloor}

\begin{document}
\title{Series-Parallel and Planar Graphs for Efficient Broadcasting}

\author{\IEEEauthorblockN{David Evangelista,
Hovhannes A. Harutyunyan, and Aram Khanlari\\}
\IEEEauthorblockA{Department of Computer Science and Software Engineering\\
Concordia University\\
Montreal, Quebec, Canada}
}
\date{November 2023}

\maketitle
\setlist[description]{listparindent=0pt, leftmargin=0pt}

\section*{Abstract}

The broadcasting problem concerns the efficient dissemination of information in graphs. In classical broadcasting, a single originator vertex initially has a message to be transmitted to all vertices. Every vertex which has received the message informs at most one uninformed neighbor at each discrete time unit.

In this paper, we introduce infinite families of series-parallel graphs with efficient broadcast times: graphs on $n$ vertices with broadcast time at most $\ceil{\log_2 n } + 1$ for any $n$, graphs on $n$ vertices with broadcast time $\floor{\frac{3 \ceil{\log_2 n}}{2}}$ and maximum degree $\ceil{\log_2 n}-1$ for any $n$, and broadcast graphs on up to $2^{k-1} + 2^{\floor{\frac{k}{2}}}$ vertices with broadcast time $k$ for any $k$. We also introduce an infinite family of planar broadcast graphs on up to $2^{k-1} + 2^{\floor{\frac{3k}{4}} - 1}$ vertices with broadcast time $k$ for any $k$, which improves the known lower bound on the maximum number of vertices in a planar broadcast graph.

\section{Introduction}

\textbf{Broadcasting.} \quad
Broadcasting refers to the transmission of a message in a connected graph from a designated vertex called the \textit{originator}, to all other vertices over a span of discrete time units. At each time unit, each vertex that has already been informed with the message transmits it to at most one uninformed neighbor.
A sequence of calls transmitting the message from an originator $v$ to all other vertices is called a \textit{broadcast scheme} for vertex $v$. The broadcast time of vertex $v$ in graph $G$, denoted $b(v, G)$, is the minimum number of time units across all broadcast schemes for $v$. The broadcast time of $G$, denoted $b(G)$, is the minimum number of time units sufficient to complete broadcasting in $G$ from any originator, i.e. $b(G) = \underset{v\in G}{\max}$ $b(v, G)$. At every time unit, the number of informed vertices increases at least by one and at most doubles, and therefore the broadcast time of any graph $G$ on $n$ vertices is bounded by $\ceil{\log n} \leq b(G) \leq n-1$ (note that throughout this paper, $ \log n$ refers to $\log_2 n$). A graph $G$ on $n$ vertices is a \textit{broadcast graph} if $b(G)=\ceil{\log n}$.

The broadcasting problem has applications in network design where the topology of a network should be optimized for efficient communication. The study of broadcasting is typically divided into two categories, the first of which consists in finding the broadcast times of given graphs. This is generally difficult, as finding the broadcast time of an arbitrary graph is NP-Hard \citep{lewis1983computers,Slater} even when restricted to graph classes such as 3-regular planar graphs \citep{middendorf1993minimum}. We focus on the second category, which consists in constructing families of graphs with structure that guarantees fast broadcast time. For general classes of graphs, many topologies that achieve efficient broadcasting are known, for example hypercubes, Knödel graphs, and other graphs (see \cite{HARUTYUNYAN202356} and the references within). We refer the reader to the following survey papers on broadcasting \citep*{FRAIGNIAUD199479, harutyunyan2013broadcasting, hedetniemi1988survey, hromkovivc1996dissemination}. We restrict our search to series-parallel and planar graphs with efficient broadcast times.

\textbf{Planar and Series-Parallel Graphs.} \quad
Planar graphs are graphs that can be represented on a plane without intersecting edges. Such a representation is called a \textit{planar embedding} of a graph. Series-parallel (SP) graphs form a subclass of planar graphs and consist of multigraphs that can be derived using a recursive sequence of operations: \textit{series extensions} and/or \textit{parallel extensions} \cite{duffin}. A \textit{series extension} adds a new vertex between two adjacent vertices and a \textit{parallel extension} adds a new edge between two adjacent vertices (see Figure \hyperref[fig1.1]{1.1}). The first operation of the construction sequence is applied on a single edge between two vertices that are called \textit{terminals}, denoted $s$ and $t$. Any SP graph (except $K_2$) can alternatively be defined as the result of a \textit{series composition} or \textit{parallel composition} of two SP graphs. Consider SP graphs $G_1$ on terminals $s_1$ and $t_1$ and $G_2$ on terminals $s_2$ and $t_2$. A \textit{series composition} of $G_1$ and $G_2$ consists in merging $s_1$ with $t_2$ or $s_2$ with $t_1$. A \textit{parallel composition} of $G_1$ and $G_2$ consists in merging $s_1$ with $s_2$ and $t_1$ with $t_2$ (see Figure \hyperref[fig1.2]{1.2}). In either case, the operation yields a new SP graph. Note that having multiple edges between pairs of vertices does not impact broadcasting, and thus throughout this paper we only consider simple graphs, i.e. SP graphs that have no multiple edges at the end of their recursive construction.

\begin{figure}[!ht]
    \centering
    \subfigure[\label{fig1.1}A sequence of two SP operations: a parallel extension followed by a series extension]{

        \resizebox{0.4\textwidth}{!}{
            \begin{tikzpicture}[ auto ,node distance =1.7 cm and 2.3cm,, thin, state/.style ={circle ,top color =white , bottom color = white , draw, black, minimum width =0.3 cm}]
        \tikzset{edge/.style = {->,> = latex'}}

            \tkzDefPoint(-1.5,1){A}
            \tkzDefPoint(-1.5,-1){B}
            \tkzDefPoint(-0.5,0){C}
            \tkzDefPoint(1,0){D}
            \tkzDefPoint(2,1.5){E}
            \tkzDefPoint(3,0){F}
            \tkzDefPoint(2,-1.5){G}

            \node[state] (1) at (A)[minimum size=0.2cm] {$s$};
            \node[state] (2) at (B)[minimum size=0.2cm] {$t$};

            \node[state] (5) at (E)[minimum size=0.2cm] {$s$};
            \node[state] (7) at (G)[minimum size=0.2cm] {$t$};

            \draw (1) -- (2);
            \draw[edge] (C) -- (D);

            \draw (5) to [out=-70,in=70] (7);
            \draw (5) to [out=250,in=110] (7);

            \tkzDefPoint(6,1.5){A}
            \tkzDefPoint(6,-1.5){B}
            \tkzDefPoint(3.5,0){C}
            \tkzDefPoint(5,0){D}
            \tkzDefPoint(7.4,0){E}
            \tkzDefPoint(11,0){F}
            \tkzDefPoint(10,-1){G}

            \node[state] (1) at (A)[minimum size=0.2cm] {$s$};
            \node[state] (2) at (B)[minimum size=0.2cm] {$t$};

            \node[state] (5) at (E)[minimum size=0.6cm] { };

            \draw (1)--(2);
            \draw (1)--(5);
            \draw[edge] (C) -- (D);
            \draw (5) -- (2);
            
        \end{tikzpicture}
        }
        }
        %
       \quad
       \subfigure[\label{fig1.2}Two SP graphs and their parallel composition resulting in a new SP graph]{
       \subfigure{
             \resizebox{0.14\textwidth}{!}{
        \begin {tikzpicture}[ auto ,node distance =1.7 cm and 2.3cm,, thin, state/.style ={circle ,top color =white , bottom color = white , draw, black, minimum width =0.3 cm}]
        \tikzset{edge/.style = {->,> = latex'}}

        \tikzset{edge/.style = {->,> = latex'}}

        \foreach \i\j in {0/2,0/3,1/3,1/4,3/3,3/4}
            \node[state] (\i\j) at (\i,\j)[minimum size=0.2cm]{};
        \node[state] (a) at (4,4)[minimum size=0.2cm]{};

        \node[state] (s1) at (2,5)[minimum size=0.2cm]{$s_1$};
        
        \draw (02) -- (03);
        \draw (14) -- (03);
        \draw (14) -- (13);
        \draw (14) -- (s1);
        \draw (34) -- (s1);
        \draw (34) -- (33);
        \draw (a) -- (s1);

        \node[state] (t) at (2,0.8)[minimum size=0.2cm]{$t_1$};
        
        \foreach \i in {02,03,13,14,33,34,s1,a}
            \draw (t) -- (\i);
         
        \tikzset{edge/.style = {->,> = latex'}}       
        \end{tikzpicture}
        }
        }
         \subfigure{
        \resizebox{0.058\textwidth}{!}{
        \begin {tikzpicture}[ auto ,node distance =1.7 cm and 2.3cm,, thin, state/.style ={circle ,top color =white , bottom color = white , draw, black, minimum width =0.3 cm}]
        \tikzset{edge/.style = {->,> = latex'}}

        \tikzset{edge/.style = {->,> = latex'}}

        \foreach \i\j in {1/3,1/4}
            \node[state] (\i\j) at (\i,\j)[minimum size=0.2cm]{};
        \node[state] (a) at (1.5,4)[minimum size=0.2cm]{};

        \node[state] (s1) at (2,5)[minimum size=0.2cm]{$s_2$};
        
        \draw (14) -- (13);
        \draw (14) -- (s1);
        \draw (a) -- (s1);

        \node[state] (t) at (2,0.8)[minimum size=0.2cm]{$t_2$};
        
        \foreach \i in {13,14,s1,a}
            \draw (t) -- (\i);
         
        \tikzset{edge/.style = {->,> = latex'}}       
        \end{tikzpicture}
        } 
        }
        %
    ~
         \subfigure{
     \resizebox{0.18\textwidth}{!}{
        \begin {tikzpicture}[ auto ,node distance =1.7 cm and 2.3cm,, thin, state/.style ={circle ,top color =white , bottom color = white , draw, black, minimum width =0.3 cm}]
        \tikzset{edge/.style = {->,> = latex'}}
        

        \tikzset{edge/.style = {->,> = latex'}}

        \foreach \i\j in {0/2,0/3,1/3,1/4,2/3,2/4,4/4, 4/3, 5/4}
            \node[state] (\i\j) at (\i,\j)[minimum size=0.2cm]{};
        \node[state] (a) at (2.5,4)[minimum size=0.2cm]{};
        \node[state] (35) at (3,5)[minimum size=0.2cm]{$s$};
        
        \draw (02) -- (03);
        \draw (14) -- (03);
        \draw (14) -- (13);
        \draw (14) -- (35);
        \draw (24) -- (35);
        \draw (24) -- (23);
        \draw (a) -- (35);
        \draw (44) -- (35);
        \draw (54) -- (35);
        \draw (44) -- (43);

        \node[state] (t) at (3,0.8)[minimum size=0.2cm]{$t$};
        
        \foreach \i in {02,03,13,14,23,24,35, 43,a,44, 54}
            \draw (t) -- (\i);
        \tikzset{edge/.style = {->,> = latex'}}
        
        
        \end{tikzpicture}
        }
        }
        \setcounter{subfigure}{2}
        }
    \label{sp_ops_alt_def}
    \setcounter{figure}{1}
\end{figure}

Properties of SP and planar graphs have been leveraged for algorithmic problems. Some decision problems that are NP-complete for arbitrary graphs can be solved in linear time for SP graphs, including maximum matching, maximum independent set, and minimum dominating set. The recognition, construction, and decomposition of SP graphs can also be done in linear time \citep{eppstein}. Planar graphs are also known to have certain structural properties that can be exploited for various algorithmic problems. We refer the reader to \citep{planarbook} for examples.

\textbf{Known Results on Broadcasting in Series-Parallel and Planar Graphs} \quad

The problems of determining broadcast times or optimal broadcast schemes are NP-complete in SP graphs since their treewidth is at most two \citep*{tale2024} as well as in planar graphs \citep{jakoby1998complexity}. Kortsarz and Peleg \cite{kortsarz1995approximation} presented a $\mathcal{O}(\log{}n)$-approximation algorithm for broadcasting in graph families with small separators, such as SP graphs. Jakoby et. al \citep{jakoby1998complexity} showed that the problem can be solved efficiently for graphs with decomposable properties. Hell and Seyffarth \cite{HellPaper} provided a family of planar broadcast graphs on up to $2^{k-1} + 2^{\floor{\frac{2k}{3}}} + 1$ vertices for any $k$, and thus established a lower bound on the maximum number of vertices in planar broadcast graphs. They also showed a trivial upper bound on the maximum number of vertices in a planar broadcast graph of at most $1+ 2^{k-1} + 2^{k-2} + 2^{k-3} + 2^{k-4} + 2^{k-5} = 2^k - 2^{k-5} + 1$ by virtue of containing a vertex of degree at most 5.

\textbf{Structure of the Paper} \quad
This paper is structured as follows. In Section 2, we examine general bounds for broadcast times of SP graphs. In Section 3, we introduce two families of SP graphs on $n$ vertices for any $n$ with logarithmic broadcast time: the first with maximum degree $\ceil{\log n} - 1$ and broadcast time $\floor{\frac{3 \ceil{\log n}}{2}}$, and the second with broadcast time $\ceil{\log n} + 1$. In Section 4, we introduce a family of SP broadcast graphs on up to $2^{k-1} + 2^{\floor{\frac{k}{2}}}$ vertices, for any $k$. In Section 5, we introduce a family of planar broadcast graphs on up to $2^{k-1} + 2^{\floor{\frac{3k}{4}}-1}$ vertices with broadcast time $k$, for any $k$, thereby improving the lower bound in \cite{HellPaper}.

\section{Broadcasting in Series-Parallel Graphs}

We first establish general bounds for the broadcast times of SP graphs. A path on $n$ vertices is an SP graph and has broadcast time $n-1$, giving a tight upper bound on the broadcast time $b(G) \leq n-1$ for any SP graph $G$. This upper bound can be improved for certain subclasses of SP graphs, such as graphs that can be generated from $K_3$. We establish the following property which, to the best of our knowledge, has not been studied in the literature.

\begin{lemma}
If $G$ is an SP graph that can be generated from $K_3$ using a sequence of SP operations, then $G$ is $2$-connected.
\end{lemma}

\begin{proof}

    Let $G$ be such a graph. To prove $G$'s $2$-connectivity, it suffices to show that there are two vertex-disjoint paths between any pair of distinct vertices in $G$. The proof is by induction on the number of SP operations in a fixed recursive construction sequence of $G$ applied on $K_3$.
    
    For $K_3$, the statement is true. Assume that after the $i$\textsuperscript{th} SP operation ($i \geq 0)$, the graph is $2$-connected. Consider operation $i+1$. Performing a parallel extension on any edge cannot remove a path between two vertices. On the other hand, performing a series extension on an edge $\{p, r\}$ introduces a new vertex, say $q$. Between any pair of distinct vertices excluding $q$, there still exist two paths, as any path previously containing edge $\{p, r\}$ can now be modified to contain $\{p, q\} \cup \{q, r\}$ instead. For any vertex $v \neq q$, there exist two vertex-disjoint paths between $q$ and $v$: $\{q, p\} \cup P_1$, where $P_1$ is the path between $p$ and $v$ not containing $\{p, r\}$, and $\{q, r\} \cup P_2$, where $P_2$ is the path between $r$ and $v$ not containing $\{r, p\}$.
\end{proof}

Every SP graph $G$ on $n$ vertices which can be generated from $K_3$ is $2$-connected and therefore has broadcast time at most $\ceil{\frac{n}{2}}$ \citep{aramThesis}. Since $G$ could be a cycle with $b(G) = \ceil{\frac{n}{2}}$, the upper bound $b(G) \leq \ceil{\frac{n}{2}}$ is tight for any graph $G$ in this subclass of SP graphs.

\section{Series-Parallel Graphs with Logarithmic Broadcast Time}

Our design of SP and planar graphs for efficient broadcasting makes use of \textit{binomial trees}, a common topology used throughout the literature.

\begin{defi} 
    A $0$-dimensional binomial tree is defined as the graph on one vertex. A \textit{$k$-dimensional binomial tree} $BT_k$ on $2^k$ vertices with $k\geq 1$ can equivalently be seen as any of the following (see Figure \ref{c2:bino}).

    \begin{enumerate}
        \item [a)] A $k$-dimensional root vertex connected to $k$ children that are themselves roots of a $BT_{k - 1},$ $BT_{k - 2},$ $...,$ and $BT_0$. Each $i$\textsuperscript{th} child root is called a \textit{subroot} of dimension $k-i$.

        \item  [b)]   For $d \in \{0, ..., k\}$, a $BT_{k-d}$ where each vertex becomes the root of an appended $BT_d$. This results in $2^{k-d} \cdot 2^d = 2^k$ vertices.
    \end{enumerate}
    \label {binodef}
\end{defi}

The following special cases of Definition \ref{binodef}\hyperref[binodef]{b} will be used throughout this paper.
\begin{itemize}
    \item When $d = 1$, every vertex in a $BT_{k-1}$ gets a new adjacent vertex, resulting in a $BT_k$.
    \item When $d = k-1$, the roots of two $BT_{k-1}$'s become adjacent, with one of them designated as the root of the resulting $BT_k$. Note however that the other root (the other vertex of degree $k$) can also be considered \textit{a} root of the $BT_k$.
\end{itemize}

\begin{obs}
\label{obs3}
As a consequence of Definition \ref{binodef}\hyperref[binodef]{b}, any vertex in a $BT_k$ is contained in some $BT_d$ subtree, for $d \in \{0, ..., k\}$ (see Figure \ref{c2:bino}).
\end{obs}

Broadcasting in binomial trees is outlined in Broadcast Scheme \ref{binoscheme}.

\begin{figure}[!ht]
     \raggedleft
    \resizebox{8cm}{!}{

    \begin{tikzpicture}[ auto , node distance =1.7 cm and 2.3cm, thin, state/.style ={circle , top color =white , bottom color = white , draw, black, minimum width =0.3 cm}]
    \tikzset{edge/.style = {->,> = latex'}}
        \foreach \i\j in {1/1,1/2,2/2,2/3,3/2,3/3,4/3,3/4,5/2,5/3,6/3,5/4,7/3,7/4,6/5}
            \node[state] (\i\j) at (\i,\j)[minimum size=0.2cm]{};
        \node[state] (84) at (8.5,4)[minimum size=0.2cm]{};
        \foreach \i in {84, 74}
            \draw (65) -- (\i);

            \draw [blue, line width=1.5] (65) -- (54); 
            \draw [blue, line width=1.5] (65) -- (34); 

        \foreach \i in {73}
            \draw (74) -- (\i);
        \foreach \i in {63,53}
            \draw (54) -- (\i);
        \foreach \i in {43,33}
            \draw (34) -- (\i);

            \draw [blue, line width=1.5] (34) -- (23); 
            
        \foreach \i in {22,12}
            \draw (23) -- (\i);
        \foreach \i in {32}
            \draw (33) -- (\i);
        \foreach \i in {52}
            \draw (53) -- (\i);
        \foreach \i in {12}
            \draw (11) -- (\i);
            
        \node (1) at (3.5,1) {$BT_3$};
        \node (2) at (5.5,1) {$BT_2$};
        \node (2) at (7.5,2) {$BT_1$};
        \node (2) at (8.5,3) {$BT_0$};
        \node[blue] (2) at (4.4, 4.9) {$BT_2$};
        
        \draw[dashed, rotate=45] (3.5,0) ellipse (2.5cm and 1.3cm);
        \draw[dashed] (5.4,3) ellipse (1cm and 1.5cm);
        
        \draw[dashed] (7,3.5) ellipse (0.5cm and 1cm);
        \draw[dashed] (8.5,4) ellipse (0.5cm and 0.5cm);
        
    \end{tikzpicture}
    }
     \raggedright
    \resizebox{8cm}{!}{
    \begin{tikzpicture}[ auto , node distance =1.7 cm and 2.3cm, thin, state/.style ={circle , top color =white , bottom color = white , draw, black, minimum width =0.3 cm}]
    \tikzset{edge/.style = {->,> = latex'}}
        \foreach \i\j in {1/1,1/2,2/2,2/3,3/2,3/3,4/3,3/4,5/2,5/3,6/3,5/4,7/3,7/4,6/5}
            \node[state] (\i\j) at (\i,\j)[minimum size=0.2cm]{};
        \node[state] (84) at (8.5,4)[minimum size=0.2cm]{};
        \foreach \i in {84,74,54,34}
            \draw (65) -- (\i);
        \foreach \i in {73}
            \draw (74) -- (\i);
        \foreach \i in {63,53}
            \draw (54) -- (\i);
        \foreach \i in {43,33,23}
            \draw (34) -- (\i);
        \foreach \i in {22,12}
            \draw (23) -- (\i);
        \foreach \i in {32}
            \draw (33) -- (\i);
        \foreach \i in {52}
            \draw (53) -- (\i);
        \foreach \i in {12}
            \draw (11) -- (\i);
        \node (1) at (1.5,4) {$BT_3$};
        \node (2) at (4.75,5) {$BT_3$};
        
        \draw[dashed, rotate=45] (3.5,0) ellipse (2.5cm and 1.2cm);
        \draw[dashed, rotate=45] (7.1, -2.2) ellipse (2.4cm and 1.9cm);
    \end{tikzpicture}
    }
    \caption{A binomial tree $BT_4$ is a root vertex with children $BT_3$, $BT_2$, $BT_1$, $BT_0$. The blue edges show that a $BT_4$ is also a $BT_2$ where each vertex is the root of an appended $BT_2$ (left). A $BT_4$ alternatively consists of two instances of $BT_3$ connected at their roots (right). }
    \label{c2:bino}
\end{figure}
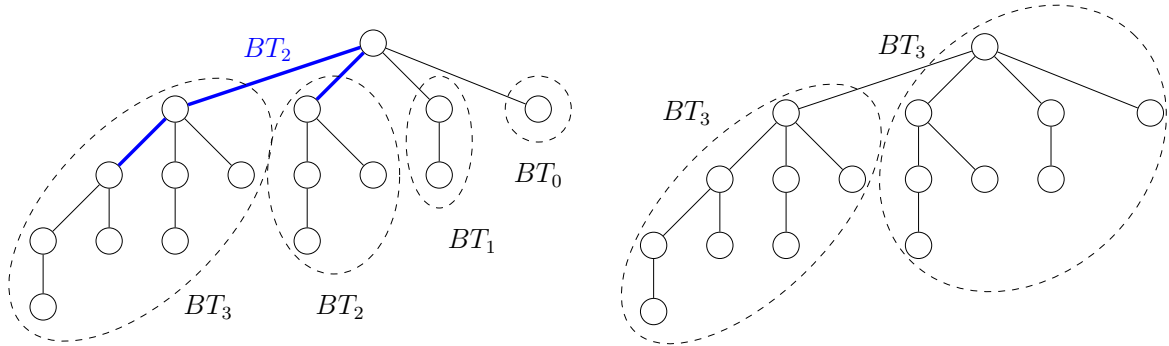

\begin{scheme}
\label{binoscheme}
    The originator informs towards the closest root (a vertex of degree $k$) which gets informed at time unit $i \geq 0$. At every time unit $i+j$ ($j \geq 1)$, every informed vertex informs its child subroot of dimension $k-j$.
\end{scheme}

\begin{obs}
\label{obs_binoroottime}
    For a root $r$ of a $BT_k$, $b(r, BT_k) = k$.
\end{obs}

\begin{obs}
\label{obs2}
    The broadcast time of a $BT_k$ is $b(BT_k) = 2k-1$. The upper bound follows from the distance from any vertex to a root being at most $k-1$ and Observation \ref{obs_binoroottime}. The lower bound follows from the diameter of the graph being $2k-1$.
\end{obs}

\subsection{Mirrored-Binomial Series-Parallel Graph}

We now show how to construct an SP graph on $n$ vertices with broadcast time at most $\floor{\frac{3\ceil{\log n}}{2}}$ and maximum degree $\ceil{\log n} - 1$, for any $n$.

A $k$-dimensional \textit{Mirrored-Binomial SP} graph $MB_k$ ($k\geq2$) on $n = 2^k$ vertices is informally described as follows. Consider two $BT_{k-1}$'s, one of them rotated by 180 degrees and positioned below the other. Each leaf in the upper tree is connected to the closest leaf in the lower tree (see Figure \ref{FMB}).

Recall that by Definition \ref{binodef}\hyperref[binodef]{b} with $d = k-2$, each of the $BT_{k-1}$'s can be seen as two $BT_{k-2}$'s, one on the left side and one on the right side, joined at the roots. Observe that for $k \geq 3$, the $BT_{k-2}$'s on the left side of the $MB_k$, in the upper and lower trees, are connected at the leaves, together forming an $MB_{k-1}$ (see to the left of the dashed line in Figure \ref{FMB}). Therefore, an $MB_k$ can alternatively be seen as two side-by-side copies of an $MB_{k-1}$, with their respective $s$ and $t$ terminals connected to each other.

\begin{theorem}
    An $MB_k$ $(k \geq 2)$ on $2^k$ vertices is SP.
\end{theorem}

\begin{proof}
The proof is by induction on $k$. Observe that an $MB_2$ is a path on four vertices and is therefore SP. For $k\geq 3$, assume that an $MB_{k-1}$ is SP. We then show that an $MB_k$ can be generated using SP operations applied on a $4$-cycle on vertices $s', s, t', t$, which is also SP. Consider edges $\{s', t\}$ and $\{s, t'\}$. Using a sequence of SP operations, an $MB_{k-1}$ can be generated from each of them since an $MB_{k-1}$ is assumed to be SP. This results in two instances of an $MB_{k-1}$ with connected terminals, i.e. an $MB_k$.
\end{proof}

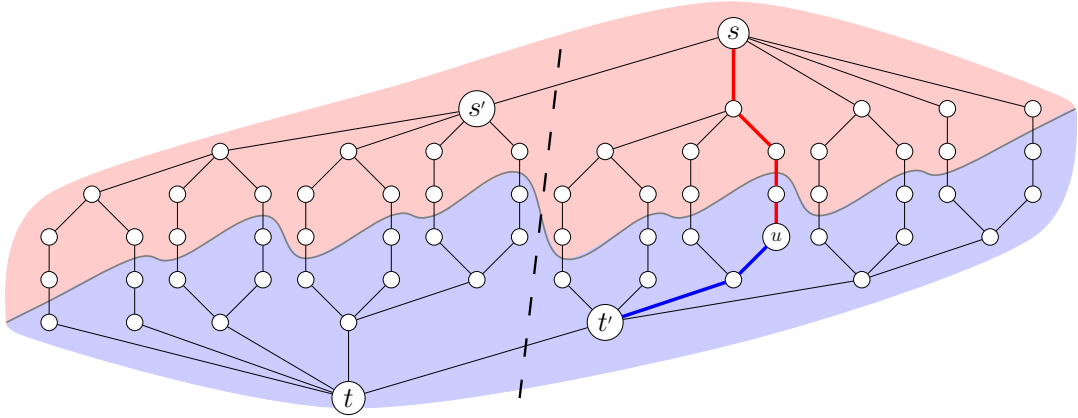
\begin{figure}[ht!]
\centering
\resizebox{15cm}{!}{
\begin {tikzpicture}[ auto ,node distance =1.7 cm and 2.3cm,, thin, state/.style ={circle ,top color =white , bottom color = white , draw, black, minimum width =0.3 cm}]
\tikzset{edge/.style = {->,> = latex'}}
\def\k{1.3}

\draw [fill=red!20, draw=red!20] plot [smooth cycle] coordinates {(\k*17,\k*8.5) (\k*9,\k*6.5) (\k,\k*4)(0,\k*1) (\k*1,\k*1.5) (\k*3,\k*2.5) (\k*4,\k*2.5) (\k*6,\k*3.5) (\k*7,\k*2.5) (\k*9,\k*3.5) (\k*10,\k*3.5) (\k*12,\k*4.5)  (\k*13,\k*2.5) (\k*15,\k*3.5) (\k*16,\k*3.5) (\k*18,\k*4.5) (\k*19,\k*3.5) (\k*21,\k*4.5) (\k*22,\k*4.5) (\k*24,\k*5.5)  (\k*24.6,\k*6.3)};

\draw [fill=blue!20, draw=blue!20] plot [smooth cycle] coordinates {(\k*8,-1*\k) (\k*0.4,\k*0.7) (\k*1,\k*1.5) (\k*3,\k*2.5) (\k*4,\k*2.5) (\k*6,\k*3.5) (\k*7,\k*2.5) (\k*9,\k*3.5) (\k*10,\k*3.5) (\k*12,\k*4.5)  (\k*13,\k*2.5) (\k*15,\k*3.5) (\k*16,\k*3.5) (\k*18,\k*4.5) (\k*19,\k*3.5) (\k*21,\k*4.5) (\k*22,\k*4.5) (\k*24,\k*5.5)  (\k*25.02,\k*5.85) (\k*24,\k*3) (\k*17,\k*0.5)  };

\draw [gray, line width=1.5pt] plot [ smooth, tension=0.7] coordinates { (0,\k*1) (\k*1,\k*1.5) (\k*3,\k*2.5) (\k*4,\k*2.5) (\k*6,\k*3.5) (\k*7,\k*2.5) (\k*9,\k*3.5) (\k*10,\k*3.5) (\k*12,\k*4.5)  (\k*13,\k*2.5) (\k*15,\k*3.5) (\k*16,\k*3.5) (\k*18,\k*4.5) (\k*19,\k*3.5) (\k*21,\k*4.5) (\k*22,\k*4.5) (\k*24,\k*5.5)  (\k*25,\k*6)};

\draw [dashed,dash pattern=on 16pt off 20pt,  line width=2pt] plot  coordinates { (\k*12,-1) (\k*13,\k*7+1)};
\foreach \i in {1,3}{
    \foreach \j in {1,2,3}{
        \node[state] (\i\j) at (\k*\i,\k*\j)[minimum size=0.5cm]{};
        }
    \draw[] (\i2) -- (\i3);
    \draw[] (\i1) -- (\i2);
    }
\foreach \i in {4,6, 7, 9, 13, 15}{
    \foreach \j in {2,3,4}{
        \node[state] (\i\j) at (\k*\i,\k*\j)[minimum size=0.5cm]{};
        }
    \draw[] (\i2) -- (\i3);
    \draw[] (\i4) -- (\i3);
    }
\foreach \i in {5,8}{
    \node[state] (\i5) at (\k*\i,\k*5)[minimum size=0.5cm]{};
    \node[state] (\i1) at (\k*\i,\k*1)[minimum size=0.5cm]{}; 
    }
\node[state] (145) at (\k*14,\k*5)[minimum size=0.5cm]{};
\node[state] (141) at (\k*14,\k*1)[minimum size=0.5cm]{\Huge$t$\Large$'$}; 
\foreach \i in {10,12, 16, 19, 21}{
    \foreach \j in {3,4,5}{
        \node[state] (\i\j) at (\k*\i,\k*\j)[minimum size=0.5cm]{}; 
        }
    \draw[] (\i4) --(\i3);
    \draw[] (\i4) --(\i5);
    }
    \node[state] (183) at (\k*18,\k*3)[minimum size=0.5cm]{\Large$u$}; 
    \node[state] (184) at (\k*18,\k*4)[minimum size=0.5cm]{}; 
    \node[state] (185) at (\k*18,\k*5)[minimum size=0.5cm]{}; 
    \draw[red, line width =3pt] (184) --(183);
    \draw[red, line width =3pt] (184) --(185);
\foreach \i in {17,20}{
    \node[state] (\i6) at (\k*\i,\k*6)[minimum size=0.5cm]{};
    \node[state] (\i2) at (\k*\i,\k*2)[minimum size=0.5cm]{}; 
    }
\node[state] (116) at (\k*11,\k*6)[minimum size=0.5cm]{\Huge$s$\Large$'$};
\node[state] (112) at (\k*11,\k*2)[minimum size=0.5cm]{}; 
\foreach \i in {22,24}{
    \foreach \j in {4,5,6}{
        \node[state] (\i\j) at (\k*\i,\k*\j)[minimum size=0.5cm]{};
        }
    \draw[] (\i4) -- (\i5);
    \draw[] (\i5) -- (\i6);
    }
\node[state] (24) at (\k*2,\k*4)[minimum size=0.5cm]{};
\node[state] (233) at (\k*23,\k*3)[minimum size=0.5cm]{};
\node[state] (t) at (\k*8,-1)[minimum size=0.5cm]{\Huge$t$};
\node[state] (s) at (\k*17,\k*7+1)[minimum size=0.5cm]{\Huge$s$};

\foreach \i\j in {44/55, 55/64, 42/51, 51/62, 74/85, 85/94, 72/81, 81/92, 134/145, 145/154, 132/141, 141/152,  105/116, 116/125, 103/112, 112/123, 165/176, 163/172,  195/206, 206/215, 193/202, 202/213, 13/24, 24/33, 224/233, 233/244, 233/202, 145/176, 112/81, 24/55, 141/202, 85/116, 55/116}{
    \draw[] (\i) --(\j);
    }
\foreach \i\j in {176/185, s/176}{
    \draw[red, line width =3pt] (\i) --(\j);
    }
    \draw[blue, line width =3pt] (141) --(172);
    \draw[blue, line width =3pt] (183) --(172);
\foreach \i\j in {t/11,t/31,t/51,t/81,t/141, s/116, s/206,s/226,s/246}{
    \draw[] (\i) --(\j);
    }



\end{tikzpicture}
}
\caption{A Mirrored-Binomial SP graph $MB_k$ can be seen as: an upper and lower $BT_{k-1}$ joined at the leaves, or alternatively, two $MB_{k-1}$ side by side, joined at their terminals. Since $dist(s, t') = k$, the path from $u$ to $s$ (in red) is of length $\lceil \frac{k}{2}\rceil +j$, and the path from $u$ to $t'$ (in blue) is of length $\lfloor \frac{k}{2}\rfloor -j$.}
\label{FMB}
\end{figure}

To show the broadcast time of an $MB_k$, we highlight the following lemma.

\begin{lemma} \label{b(s,MB_k)}
    For an $MB_k$ and $u \in \{s, s', t, t'\}$, $b(u, MB_k) = k + 1$.
\end{lemma}

\begin{proof}

Without loss of generality, assume $u = s$.

The lower bound follows from the fact that $dist(s, t) = k+1$.

To prove the upper bound, consider the following broadcast scheme for $s$. 
\begin{scheme}
\label{binoschemeMB}
    Broadcast Scheme \ref{binoscheme} is used to finish informing the $BT_{k-1}$ rooted at $s$ at time $k-1$. At time $k$, the vertices that were informed at time $k-1$ inform the leaves of the $BT_{k-1}$ rooted at $t$. At time $k+1$, the vertices that were informed at time $k$ inform the remaining uninformed vertices, i.e. their parents in the $BT_{k-1}$ rooted at $t$.
\end{scheme}
\end{proof}

\begin{theorem}
    The broadcast time of an $MB_k$ on $n = 2^k$ vertices is $b(MB_k)= \left\lfloor \frac{3 \log n} {2} \right\rfloor$.
\end{theorem}

\bigbreak

\begin{proof}

Observe that $\log n = k$. To prove the upper bound, we show there exists a broadcast scheme for any originator that completes broadcasting by time $\floor{\frac{3k}{2}}$.

Assume without loss of generality that the originator $u$ is on the right side of the graph, between $s$ and $t'$ with $dist(s, u) \geq dist(u, t')$ (see Figure \ref{FMB}). Then, it suffices to show that the vertices between $s'$ and $t$, i.e. the left side of the graph, are informed by time $\floor{\frac{3k}{2}}$. This follows from the fact that any path from $u$ to the left side of the graph must go through $s'$ or $t$, and it can be guaranteed that $s$ and $t'$ are informed no later than $s'$ and $t$. Note that the following broadcast scheme prioritizes informing terminals as early as possible.

Since the distance between $s$ and $t'$ is $k$, the distances $dist(s, u)$ and $dist(u, t')$ can be written in terms of some $j$:

\begin{equation} \label{MB_k_dist_eq}
dist(s, t') = dist(s, u) + dist(u, t') = (\left\lceil{\frac{k}{2}}\right\rceil + j) + (\left\lfloor\frac{k}{2}\right\rfloor-j)
\end{equation}

Assume that in the first time unit, $u$ informs towards $s$ and from the second time unit towards $t'$. Then, $s'$ gets informed at time $\ceil{\frac{k}{2}} + j + 1$, $t'$ gets informed at time $\floor{\frac{k}{2}}-j+1$, and $t$ gets informed at time $\floor{\frac{k}{2}}-j+2$ (see Figure \ref{FMB}). Once $s'$ and $t$ are informed, they conduct Broadcast Scheme \ref{binoscheme} in their $BT_{k-2}$ subtrees. At time $\floor{\frac{3k}{2}}$, vertex $s'$ will have informed its binomial subtree of dimension $\floor{\frac{3k}{2}} - (\ceil{\frac{k}{2}} + j + 1) \geq k-j-2$. At time $\floor{\frac{3k}{2}}-j$, vertex $t$ will have informed its $BT_{k-2}$ subtree.

\begin{itemize}
    \item If $j = 0$, broadcasting completes at time $\floor{\frac{3k}{2}}$.
    \item If $j = 1$, the $BT_{k-2}$ subtree rooted at $t$ is informed at time $\floor{\frac{3k}{2}}-1$. In the following time unit, its leaves inform the leaves of the $BT_{k-2}$ subtree rooted at $s'$. In this case, broadcasting completes at time $\floor{\frac{3k}{2}}$.
    \item If $j \geq 2$, vertex $t'$ gets informed at time at most $\floor{\frac{k}{2}}-1$, and by Lemma \ref{b(s,MB_k)}, broadcasting completes by time $\floor{\frac{3k}{2}}$.
\end{itemize}

To prove the lower bound on the broadcast time of an $MB_k$, we show the existence of an originator $u$ such that $b(u, MB_k) \geq \floor{\frac{3k}{2}}$. Let $u$ be a vertex between $s$ and $t'$ such that $j=0$ in Equation \ref{MB_k_dist_eq}. It has already been established that if $u$ informs towards $s$ in the first time unit, the graph is informed by time $\floor{\frac{3k}{2}}$. With this strategy, the left side of the $MB_k$ cannot be informed sooner as $s$ and $t'$ inform $s'$ and $t$ respectively as early as possible. If instead $u$ uses a different strategy and informs towards $t'$ in the first time unit, then $s$ and $t$ cannot get informed at time sooner than $\floor{\frac{k}{2}}+1$, and by Observation \ref{obs_binoroottime} both require $k-1$ additional time units to finish broadcasting in their $BT_{k-1}$'s. This results in $\floor{\frac{k}{2}}+1 +k-1= \floor{\frac{3k}{2}}$ time units to complete broadcasting.

Combining the upper and lower bounds, we get that $b(MB_k) = \floor{\frac{3k}{2}}$.
\end{proof}

Variations of an $MB_k$ can be constructed by removing up to $2^{k-1} -1$ leaves before joining the two $BT_{k-1}$'s, which does not increase the broadcast time. Such constructions provide SP graphs on $n$ vertices with broadcast time at most $\floor{\frac{3\ceil{\log n}}{2}}$ and maximum degree $\ceil{\log n} - 1$, for any $n$. This suggests that network topologies based on these SP graphs are less prone to fault tolerance issues as opposed to topologies that rely on vertices of large degree to achieve efficient broadcasting.

\subsection {Relaxed Broadcast Series-Parallel Graph}

A question arises as to whether there exist SP broadcast graphs on $n$ vertices for any $n$. The following theorem shows that this is not the case.

\begin{theorem}
\label{thrm_nobg2k}
There does not exist an SP broadcast graph on $2^k$ vertices for $k \geq 3$.
\end{theorem}

\begin{proof}

Consider an SP graph on $2^k$ vertices with $k \geq 3$. Such a graph must have a sequence of SP operations for its construction with at least six series extensions, as it has at a minimum of eight vertices. The last vertex introduced in a series extension, say $v$, must have only two neighbors. As an originator, $v$ does not have any vertices to inform after time 2. Therefore, when originating from $v$, the maximum number of vertices that can be informed in $k$ time units is at most $1 + 2^{k-1} + 2^{k-2}$, which is strictly less than $2^k$ when $k \geq 3$.
\end{proof}

\begin{theorem}

For any $n$, there exists an SP graph on $n$ vertices with broadcast time at most $\ceil{\log n} + 1$.

\end{theorem}

\begin{proof}
    For some $n$ and $k = \ceil{\log n}$, consider a $BT_k$ rooted at vertex (and terminal) $t$ and denote its children of decreasing degrees by $r_{k}, r_{k-1}, ..., r_1$ respectively. Designate $r_1$ as terminal $s$. Add edges $\{r_i, r_{i-1}\}$ for all $i\in \{2,...,k\}$ (see the blue edges in Figure \ref{CBT_k_SSP}). Delete $2^k-n$ leaves to obtain the desired number of vertices $n$ and connect all vertices to $t$ (see Figure \ref{CBT_k_SSP}). The resulting graph can easily be shown to be $SP$ by induction on $k$ using upcoming Theorem \ref{Bk}. As the edges added to the $BT_k$ ensure that every vertex is at distance at most 1 from root $t$, Broadcast Scheme \ref{binoscheme} can be used to complete broadcasting by time $\ceil{\log n} + 1$ regardless of the originator.
\end{proof}

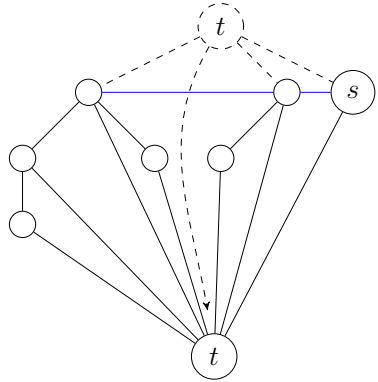
\begin{figure}[!ht]
    \centering
    \resizebox{5cm}{!}{%
        \centering
        \begin{tikzpicture}[ auto ,node distance =1.7 cm and 2.3cm, thin, state/.style ={circle ,top color =white , bottom color = white , draw, black, minimum width =0.3 cm}]
        \tikzset{edge/.style = {->,> = latex'}}
            \foreach \i\j in {1/1,1/2,3/2,2/3,4/2,5/3}
                \node[state] (\i\j) at (\i,\j)[minimum size=0.2cm]{};
            \node[state] (t) at (3.9,-1)[minimum size=0.2cm]{$t$};
            \node[state, dashed] (44) at (4,4)[minimum size=0.2cm]{$t$};
            \node[state] (63) at (6,3)[minimum size=0.2cm]{$s$};
            \foreach \i in {63,53,23}
                \draw (t) -- (\i);
            \foreach \i in {32,12}
                \draw (23) -- (\i);
            \foreach \i in {42}
                \draw (53) -- (\i);
            \foreach \i in {12}
                \draw (11) -- (\i);
            \foreach \i in {11,12,32,42}
                \draw (t) -- (\i);
            \draw[blue] (23) -- (53);
            \draw[blue] (63) -- (53);
            \draw[dashed] (44) -- (53);
            \draw[dashed] (44) -- (23);
            \draw[dashed] (44) -- (63);
            \draw [->,> = stealth', dashed] (44) to [out=-120,in=100] (3.8,-0.3);
        \end{tikzpicture}
    }
    \caption{An SP graph on $n = 2^k$ vertices with broadcast time $\ceil{\log n} + 1$.}
    \label{CBT_k_SSP}
\end{figure}

\section{Series-Parallel Broadcast Graphs}

We have established that there exist SP graphs on $n$ vertices for any $n$ with broadcast time at most $\ceil{\log n} + 1$. We have also established Theorem \ref{thrm_nobg2k}, which stated that for $k \geq 3$, there do not exist SP broadcast graphs for $n > 1 + 2^{k-1} + 2^{k-2}$. We now study the following question: what is the largest integer $n$ for which there exists an SP broadcast graph on $n$ vertices?

\subsection{Binomial Series-Parallel Graph} \label{section_BSP_k}

We define a $k$-dimensional \textit{Binomial SP} graph $B_k$ on $2^k+1$ vertices as a $BT_k$ rooted at a vertex $s$ with its $2^k$ vertices connected to an additional vertex $t$ (see Figure \ref{DB_k_SSP_proof}).

\begin{theorem} \label{Bk}
    For every $k$, a $B_k$ is an SP broadcast graph on $2^{k}+1$ vertices.
\end{theorem}

\begin{proof}

We show that a $B_k$ is SP with terminals $s$ and $t$.

For $k = 0$, a $B_k$ is a single edge $\{s, t\}$. For $k \geq 1$, $B_{k}$ is obtained as follows. For every vertex $v$ and edge $e = \{v, t\}$ in a $B_{k-1}$, perform a parallel extension followed by a series extension on $e$. Recalling Definition \ref{binodef}\hyperref[binodef]{b} with $d = 1$, we note that for every vertex in the $BT_{k-1}$, a new adjacent vertex is added and connected to $t$, thus resulting in a $B_k$. 

This graph was introduced and shown to be a planar broadcast graph in \cite{HellPaper}.

\begin{figure}[!ht]
    \centering
    
    \subfigure{
    \resizebox{.15\linewidth}{!}{
    \begin{tikzpicture}[ auto ,node distance =1.7 cm and 2.3cm, thin, state/.style ={circle ,top color =white , bottom color = white , draw, black, minimum width =0.3 cm}, scale=.5]
    \tikzset{edge/.style = {->,> = latex'}}

        \foreach \i\j in {4/1,5/3,7/3}
            \node[state] (\i\j) at (2*\i,\j)[minimum size=0.2cm]{};
        \node[state] (64) at (12,4)[minimum size=0.2cm]{$s$};
        \foreach \i in {73,53}
            \draw (64) -- (\i);
        \foreach \i in {41}
            \draw (53) -- (\i);

        \node[state] (t) at (12,-5)[minimum size=0.2cm]{ $t$};
        
        \foreach \i in {41,53,73,64}
            \draw (t) -- (\i);
        
    \end{tikzpicture}
    }
    }
    \subfigure{
    \resizebox{.15\linewidth}{!}{
    \begin{tikzpicture}[ auto ,node distance =1.7 cm and 2.3cm, thin, state/.style ={circle ,top color =white , bottom color = white , draw, black, minimum width =0.3 cm}, scale=.5]
    \tikzset{edge/.style = {->,> = latex'}}

        \foreach \i\j in {4/1,5/3,7/3}
            \node[state] (\i\j) at (2*\i,\j)[minimum size=0.2cm]{};
        \node[state] (64) at (12,4)[minimum size=0.2cm]{ $s$};
        \foreach \i in {73,53}
            \draw (64) -- (\i);
        \foreach \i in {41}
            \draw (53) -- (\i);

        \node[state] (t) at (12,-5)[minimum size=0.2cm]{ $t$};
        
        \foreach \i in {41,64}
            \draw (\i) to [bend right=6]  (t);
         \foreach \i in {41,64}
            \draw (\i) to [bend left=6]  (t);
        \foreach \i in {53,73}
            \draw (\i) to [bend left=6]  (t);
         \foreach \i in {53,73}
            \draw (\i) to [bend right=6]  (t);
    \end{tikzpicture}
    }
    }
    \subfigure{
    \resizebox{.3\linewidth}{!}{
    
    \begin{tikzpicture}[ auto ,node distance =2 cm and 2.3cm, thin, state/.style ={circle ,top color =white , bottom color = white , draw, black, minimum width =0.3 cm}]
    \tikzset{edge/.style = {->,> = latex'}}
         \foreach \i\j in {1/1,1/2,3/2,2/3,4/2,5/3,7/3}
            \node[state] (\i\j) at (\i,\j)[minimum size=0.2cm]{};
        \node[state] (64) at (6,4)[minimum size=0.2cm]{$s$};
        \foreach \i in {73,53,23}
            \draw (64) -- (\i);
        \foreach \i in {32,12}
            \draw (23) -- (\i);
        \foreach \i in {42}
            \draw (53) -- (\i);
        \foreach \i in {12}
            \draw (11) -- (\i);

        \node[state] (t) at (3.9,-1)[minimum size=0.2cm]{$t$};
        
        \foreach \i in {11,12,32,23,42,53,73,64}
            \draw (t) -- (\i);
    \end{tikzpicture}
    }
    }
    \caption{A Binomial SP graph $B_3$ on $2^3 + 1$ vertices generated from a $B_2$ using a sequence of SP operations.}
    \label{DB_k_SSP_proof}
\end{figure}
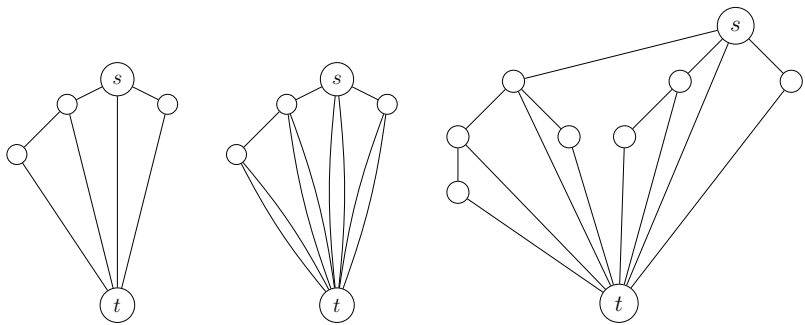
\end{proof}

\subsection{Series-Parallel Broadcast Graph of Larger Size} \label{SP_construction}

We now introduce an SP broadcast graph on $2^{k-1} + 2^{\floor{\frac{k}{2}}}$ vertices, for any $k$. A $k$-dimensional \textit{Extended-Binomial SP} graph $EB_k$ is an SP graph obtained from the parallel composition of SP graphs $G_1 = B_{k-1}$ and $G_2 = B_{\floor{\frac{k}{2}}}$ (see the last graph in Figure \hyperref[fig1.2]{1.2} for $EB_4$).

\begin{theorem} \label{thrm_ebk}
{
There exists an SP broadcast graph on $n$ vertices with $2^{k-1} + 1 \leq n \leq 2^{k-1} + 2^{\floor{\frac{k}{2}}}$, for any $k$.
}
\end{theorem}

\begin{proof}
We show that an $EB_k$ on $n = 2^{k-1} + 2^{\floor{\frac{k}{2}}}$ vertices is an SP broadcast graph.

An $EB_k$ is SP by definition. Recall that the broadcast time of any graph on $n$ vertices is at least $\ceil{\log n}$. We show that $b(EB_k) = \ceil{\log n} = k$ by providing a broadcast scheme for any vertex.

If $t$ is the originator, it informs some arbitrary vertex $v$. Otherwise, the originator, then denoted $v$, first informs $t$. 
Beginning at time unit 2, vertex $v$ informs the $BT_{\floor{\frac{k}{2}}}$ subtree that it is contained in (by Observation \ref{obs3}) and finishes doing so by time $k$ (by Observation \ref{obs2}). Note that the $BT_{\floor{\frac{k}{2}}}$ subtree that $v$ is contained in may not necessarily be $G_2 \setminus \{t\}$, as it may instead be a subtree of $G_1 \setminus \{t\}$.

At every time unit $i$ ($2\leq i\leq k$), vertex $t$ informs $r_i$, the root of the $BT_{k-i}$ subtree which does not contain $v$ (see Figure \ref{f2}). Once informed, every $r_i$ starts conducting Broadcast Scheme \ref{binoscheme} to complete broadcasting in its $BT_{k-i}$ subtree by time $k$. The total number of vertices informed by time $k$ is then $1 + 2^{\floor{\frac{k}{2}}} + \sum\limits_{i=0}^{k-2} 2^{k-i} = 2^{k-1} + 2^{\floor{\frac{k}{2}}}$, and this concludes the broadcast scheme.

To obtain a SP broadcast graph on $n$ vertices with $2^{k-1} < n < 2^{k-1} + 2^{\floor{\frac{k}{2}}}$, we can remove any number of vertices from $G_2 \setminus\{s, t\}$ starting at the leaves without increasing the broadcast time of the graph.
\end{proof}

\begin{figure}[ht!]
\centering
\resizebox{10cm}{!}{
\begin {tikzpicture}[ auto ,node distance =1.7 cm and 2.3cm,, thin, state/.style ={circle ,top color =white , bottom color = white , draw, black, minimum width =0.3 cm}]
\tikzset{edge/.style = {->,> = latex'}}
\draw [dashed] plot coordinates {(-1,5.7) (-2.3,5) (-3.7,2.7) (-0.3,3) (-1,5.7) };
\node (spu) at (-2.7,5.3) {$BT_{k-2}$};
\node (spu) at (-1.25, 3.8) {\footnotesize $BT_{k-3}$};
\node (bttt) at (1.93,5.55) {\footnotesize$BT_{\lfloor\frac{k}{2}\rfloor}$};
\node[state] (sp) at (-1.2,5.3)[minimum size=0.2cm] {};
\node[state] (spp) at (-2.25,4.6)[minimum size=0.2cm] {};

\node[fill,circle,scale=0.25, color=red] (v) at (-2.8,3) {$v$};
\node[state,scale=0.5] (v2) at (-2.58,3.37) {};
\node[state,scale=0.5] (v3) at (-2.4,4) {};
\node (ts) at (1,4.5) {$BT_{k-2}$};
\node[state,scale=0.6] (s) at (1,6)[minimum size=0.10cm] {$s$};

\node[state,scale=0.6] (t) at (-1,1)[minimum size=0.10cm] {$t$};

\node[text=blue] (1) at (-1.8,1.7) {$1$};
\node[text=blue] (2) at (0.1,5.4) {$2$};
\node[text=blue] (2u) at (-2.52,3.15) {\tiny $2$};
\node[text=blue] (3) at (-0.5,5) {$3$};
\draw(spp) -- (-3.4,3) -- (-2.2,3) -- (spp);
\draw[red, line width =0.8pt] (-3.08,3)-- (-2.42,4)-- (-2.38,4)-- (-2.2,3)-- (-3.08,3);
\node[state, scale=0.5] (v3) at (-2.4,4) {};
\draw(sp) -- (-1.9,3.5) -- (-0.7,3.5) -- (sp);
\draw(s) -- (0,4) -- (2,4) -- (s);
\draw(s) -- (2,5) -- (2.7,5.7) -- (s);
\foreach \i in {0,2,...,12}
    \pgfmathsetmacro{\y}{\i/10}
    \pgfmathsetmacro{\a}{\y-3.4}
    \draw[color=darkgray, out = 180-\i, in = -90] (t) to (\a,3);
 \foreach \i in {0,2,...,20}
    \pgfmathsetmacro{\y}{\i/10}
    \pgfmathsetmacro{\a}{2-\y}
    \draw[color=darkgray, out = 10+\i, in = -90] (t) to (\a,4);
 \foreach \i in {0,2,...,12}
    \pgfmathsetmacro{\y}{\i/10}
    \pgfmathsetmacro{\a}{-1.9+\y}
    \draw[color= darkgray, out = 100-\i, in = -90] (t) to (\a,3.5);
\foreach \i in {0,2,...,14}
    \pgfmathsetmacro{\y}{\i/20}
    \pgfmathsetmacro{\a}{2+\y}
    \pgfmathsetmacro{\b}{5+\y}
    \draw[color=darkgray] (t) .. controls (\a-1.5+\y,\b-4-\y) and (\a+1.5,\b-2).. (\a,\b);


\draw[] (sp) to (s);
\draw[] (sp) to (spp);
\tikzset{edge/.style = {->,> = latex'}}

\draw[edge, color=blue, in = 180-8, out = -90] (-2.8,3) to (t);
\draw [edge, color=blue] (t) .. controls (0.5,2) and (-0.8,4.5).. (s);
\draw [edge, color=blue] (t) .. controls (0,2) and (0,4).. (sp);
\draw[edge, color=blue] (v) to (v2);
\draw[edge, color=blue, dashed] (v2) to (v3);
\node[fill,circle,scale=0.25, color=red] (v) at (-2.8,3) {$v$};
\end{tikzpicture}
}
\caption{Beginning of the broadcast scheme for a vertex $v$ (in red) in an Extended-Binomial SP graph $EB_k$. At time 2, the originator begins informing the  $BT_{\floor{\frac{k}{2}}}$ subtree that it is contained in (the red triangle). Meanwhile at each time $i$, vertex $t$ informs $r_{k-i}$, the root of the $BT_{k-i}$ subtree that does not contain $v$.}
\label{f2}
\end{figure}
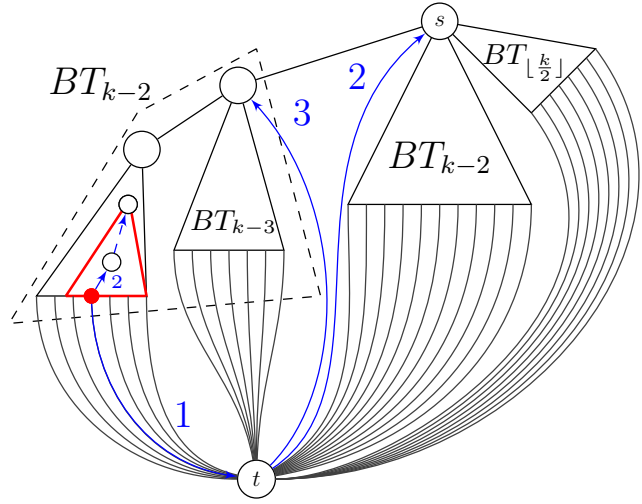

\section{Planar Broadcast Graphs}

The maximum number of vertices in a planar broadcast graph is an open question, yet some bounds are known. It is common knowledge that any planar graph contains a vertex of degree at most 5. Using this fact and a similar argument to the one in the proof of Theorem \ref{thrm_nobg2k}, Hell and Seyffarth \cite{HellPaper} showed that the maximum number of vertices in a planar broadcast graph is at most $1 + 2^{k-1} + 2^{k-2} + 2^{k-3} + 2^{k-4} + 2^{k-5}$. For a lower bound, they presented planar broadcast graphs on up to $2^{k-1} + 2^{\floor{\frac{2k}{3}}} + 1$ vertices with broadcast time $k$, for any $k$. To the best of our knowledge, no planar broadcast graph on a greater number of vertices is known. In this section, we present planar broadcast graphs on up to $2^{k-1} + 2^{\floor{\frac{3k}{4}}-1}$ vertices with broadcast time $k$, for any $k$. To do so, we modify the SP broadcast graph $EB_k$ presented in Section \ref{SP_construction} to obtain the desired graph, which is no longer SP but is still planar.

We first present a graph that is useful for the modification.
A $k$-dimensional \textit{Accelerated-Binomial SP} graph $AB_k$ on $2^k + 1$ vertices is a $B_k$ with added edges.  
An $AB_k$ is obtained using the construction steps of a $B_k$ outlined in Section \ref{section_BSP_k} with the following additional steps. Upon starting the construction with edge $\{s, t\}$, vertex $s$ is labelled $0$ to indicate its distance to itself. 
Recall that at every step $i \geq 1$ in the construction of a $B_k$, every edge incident to $t$ receives a parallel extension and a series extension resulting in a new vertex. For every such new vertex $v$ at step $i$:

\begin{itemize}

\item Connect $v$ to its ancestor labelled $\ceil{\frac{i}{3}} - 1$ in the $BT_i$ subtree rooted at $s$, if they are not already adjacent. We refer to such an added edge as a \textit{shortcut edge}.
\item Label $v$ with $\ceil{\frac{i}{3}}$ to indicate its distance to vertex $s$.
\end{itemize}

\begin{lemma} \label {abk_planar_lemma}
An $AB_k$ is planar. For $k\geq 2$, it is not SP.
\end{lemma}

\begin{proof}

    We prove an $AB_k$'s planarity by induction on $k$. For an $AB_k$ with $k \leq 3$, the graph has a planar embedding, as shown in Figure \ref{AB_3}.

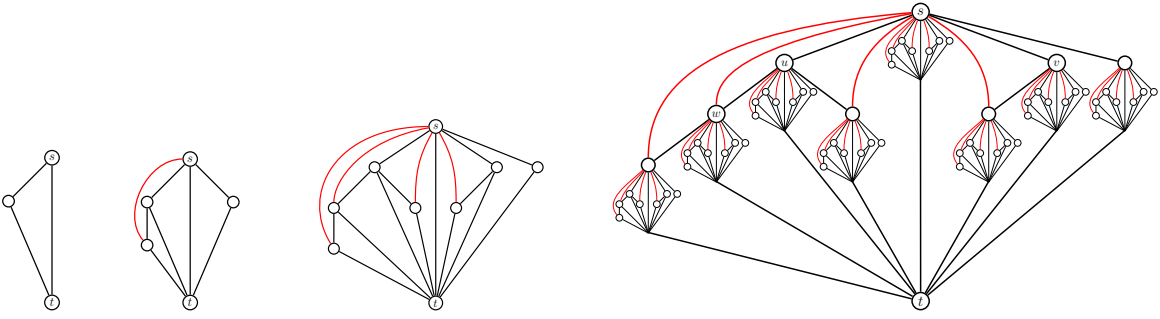
\begin{figure}[!ht]
    \centering
    \subfigure{
    \resizebox{.05\linewidth}{!}{

    \begin{tikzpicture}[ auto ,node distance =1.7 cm and 2.3cm, thin, state/.style ={circle ,top color =white , bottom color = white , draw, black,line width=2.5pt, minimum width =5pt}]
    \tikzset{edge/.style = {->,> = latex'}}
        \node[state] (12) at (-7,6)[minimum size=0.8cm]{};
        \node[state] (23) at (-4,9)[minimum size=1cm]{\Huge$s$};
        \draw [line width=2.5pt] (12) -- (23);      
        \node[state] (t) at (-4,-1)[minimum size=1cm]{\Huge$t$};
        \draw [-, line width=2.5pt, black](t)--(12);
        \draw [-, line width=2.5pt, black](t)--(23);
    \end{tikzpicture}
    }
    }
    ~
    \subfigure{
    \resizebox{.1\linewidth}{!}{

    \begin{tikzpicture}[ auto ,node distance =1.7 cm and 2.3cm, thin, state/.style ={circle ,top color =white , bottom color = white , draw, black,line width=2.5pt, minimum width =5pt}]
    \tikzset{edge/.style = {->,> = latex'}}
        \foreach \i\j in {1/1,1/2,3/2}
            \node[state] (\i\j) at (3*\i,3*\j)[minimum size=0.8cm]{};
        \node[state] (23) at (6,9)[minimum size=1cm]{\Huge$s$};
        \foreach \i in {32,12}
            \draw [line width=2.5pt] (23) -- (\i);
        \foreach \i in {12}
            \draw [line width=2.5pt] (11) -- (\i);      
        \node[state] (t) at (6,-1)[minimum size=1cm]{\Huge$t$};

        \foreach \i in {11,12,32,23}
            \draw [-, line width=2.5pt, black](t)--(\i);
        \draw [line width=2.5pt, color=red] (11) to [edge,  out=130, in =-180] (23);

    \end{tikzpicture}
    }
    }
    ~
    \subfigure{
    \resizebox{.2\linewidth}{!}{

    \begin{tikzpicture}[ auto ,node distance =1.7 cm and 2.3cm, thin, state/.style ={circle ,top color =white , bottom color = white , draw, black,line width=2.5pt, minimum width =5pt}]
    \tikzset{edge/.style = {->,> = latex'}}

        \foreach \i\j in {1/1,3/2,4/2,6/3}
            \node[state] (\i\j) at (3*\i,3*\j)[minimum size=0.8cm]{};
        \node[state] (64) at (10.5,12)[minimum size=1cm]{\Huge$s$};
        \node[state] (12) at (3,6)[minimum size=0.8cm]{};
        \node[state] (23) at (6,9)[minimum size=0.8cm]{};
        \node[state] (53) at (15,9)[minimum size=0.8cm]{};

        \foreach \i in {63,53,23}
            \draw [line width=2.5pt] (64) -- (\i);
        \foreach \i in {32,12}
            \draw [line width=2.5pt](23) -- (\i);
        \foreach \i in {42}
            \draw[line width=2.5pt] (53) -- (\i);
        \foreach \i in {12}
            \draw[line width=2.5pt] (11) -- (\i);
            
        \node[state] (t) at (10.5,-1)[minimum size=1cm]{\Huge$t$};

        \foreach \i in {11,12,32,23,42}
            \draw [-, line width=2.5pt, black](t)--(\i);
        \draw [line width=2.5pt, black] (t) to (53); 
        \draw [-, line width=2.5pt, black] (t) to (64);
        \draw [-, line width=2.5pt, black] (t) to (63);
        \draw [line width=2.5pt, color=red] (11) .. controls (0.5,7) and (3, 12).. (64);
        \draw [line width=2.5pt, color=red] (12) .. controls (3,8) and (5, 11).. (64);
        \draw [line width=2.5pt, color=red] (32) to [edge,  out=90, in =-120] (64);
        \draw [line width=2.5pt, color=red] (42) to [edge,  out=90, in=-60] (64);

    \end{tikzpicture}
    }
}
    ~
    \subfigure{
    \resizebox{.45\linewidth}{!}{

    \begin{tikzpicture}[ auto ,node distance =1.7 cm and 2.3cm,line width=1pt, state/.style ={circle ,top color =white , bottom color = white , draw, black, minimum width =1 cm}]
    \tikzset{edge/.style = {->,> = latex'}}

        \node[state, line width=2.5pt] (32) at (12,6)[minimum size=0.8cm]{};
        \node[state, line width=2.5pt] (52) at (20,6)[minimum size=0.8cm]{};
        \node[state, line width=2.5pt] (73) at (28,9)[minimum size=0.8cm]{};
        \node[state, line width=2.5pt] (44) at (16,12)[minimum size=1cm]{\Huge$s$}; 
        \node[state, line width=2.5pt] (12) at (4,6)[minimum size=1cm]{\huge$w$};
        \node[state, line width=2.5pt] (23) at (8,9)[minimum size=1cm]{\huge$u$};
        \node[state, line width=2.5pt] (63) at (24,9)[minimum size=1cm]{\huge$v$};
        \node[state, line width=2.5pt] (01) at (0,3)[minimum size=0.8cm]{};

        \foreach \i in {73,63,23}
            \draw [line width=2.5pt] (44) -- (\i);
        \foreach \i in {32,12}
            \draw [line width=2.5pt](23) -- (\i);
        \foreach \i in {52}
            \draw[line width=2.5pt] (63) -- (\i);
        \foreach \i in {12}
            \draw[line width=2.5pt] (01) -- (\i);

        \node[state,line width=2.5pt] (t) at (16,-5)[minimum size=1cm]{\Huge$t$};

        \foreach \k\p in {1/2,0/1}
        {
            \coordinate[draw=none] (inv\k\p) at (4*\k, 3*\p-4)[minimum size=0cm]{};
            \draw [-{Round Cap []}, line width=2.5pt, black,shorten >=-0.5mm](t) to (4*\k-0.01, 3*\p-3.965);
            }
        \foreach \k\p in { 3/2, 2/3, 4/4, 5/2, 6/3,7/3}
        {
            \coordinate[draw=none] (inv\k\p) at (4*\k, 3*\p-4)[minimum size=0cm]{};
            \draw [-{Round Cap []}, line width=2.5pt, black,shorten >=-0.5mm](t) to (4*\k-0.01, 3*\p-3.975);
            }
        \foreach \k\p in {3/2, 1/2, 2/3, 0/1,4/4, 5/2, 6/3,7/3}
        {
            \foreach \i\j in {1/2,3/2,2/3}
                \node[state] (\k\p\i\j) at (4*\k-2.3+0.6*\i,3*\p-3.5+0.6*\j)[minimum size=0.1cm]{};
            \foreach \i\j in {4/2,5/3,6/3}
                \node[state] (\k\p\i\j) at (4*\k-1.9+0.6*\i,3*\p-3.5+0.6*\j)[minimum size=0.1cm]{};
            \node[state] (\k\p10) at (4*\k-1.7,3*\p-3.1)[minimum size=0.1cm]{};
        \foreach \i in {\k\p63,\k\p53,\k\p23}
            \draw [] (\k\p) -- (\i);
        \foreach \i in {\k\p32,\k\p12}
            \draw [] (\k\p23) -- (\i);
        \foreach \i in {\k\p42}
            \draw [] (\k\p53) -- (\i);
        \foreach \i in {\k\p12}
            \draw [] (\k\p10) -- (\i);
        \draw [color=red] (\k\p10) to  [edge,out=130, in=-137] (\k\p);
        \draw [color=red] (\k\p12) to[edge,out=70, in=-130] (\k\p);
        \draw [color=red] (\k\p32) to [edge,  out=90, in =-110] (\k\p);
        \draw [color=red] (\k\p42) to [edge,  out=90, in =-70] (\k\p);

        \foreach \i in {\k\p10,\k\p12,\k\p32,\k\p23,\k\p53,\k\p63, \k\p,\k\p42}
            \draw [-, line width=1pt, black](inv\k\p) -- (\i);
        }
        
        \draw [line width=2.5pt, color=red] (01) .. controls (0,8) and (4, 11).. (44);
        \draw [line width=2.5pt, color=red] (12) .. controls (4,8) and (5, 10).. (44);
        \draw [line width=2.5pt, color=red] (32) to [edge,  out=90, in =-150] (44);
        \draw [line width=2.5pt, color=red] (52) to [edge,  out=90, in=-30] (44);
    \end{tikzpicture}
    }
    }
    
    \caption{Planar embeddings of Accelerated-Binomial SP graphs $AB_k$ for $k \in \{1, 2, 3, 6\}$. An $AB_k$ is a $B_k$ on $2^k + 1$ vertices with added shortcut edges (in red) such that every vertex in an $AB_k \setminus \{t\}$ is at distance at most $\ceil{\frac{k}{3}}$ from $s$.}
    \label{AB_3} 
    \end{figure}

    Assume that an $AB_{3m}$ is planar for some $m \geq 1$. Then, an $AB_{3m+j}$ with $j \in \{1, 2, 3\}$ is obtained as follows. For every vertex $v$ in the $AB_{3m}$, replace any edge $\{v, t\}$ with an $AB_j$ (see Figure \ref{AB_3}). The resulting graph is an $AB_{3m+j}$ and is planar, as shown by the planar embedding. 
    
    To show that a graph is not SP it suffices to show that it contains $K_4$ as a minor \cite{duffin}. An $AB_k$ ($k \geq 2$) is therefore not SP as it contains $K_4$ as a subgraph (see the subgraph induced by vertex set $\{s, u, w, t\}$ in Figure \ref{AB_3}).
\end{proof}

\begin{lemma}
\label{btimeABk}
    The broadcast time of an $AB_k \setminus \{t\}$ is at most $\ceil{\frac{4k}{3}}$.
\end{lemma}

\begin{proof}
    Assume the originator has label $i$. A path towards $s$ can be informed via the vertices with labels $[{i-1}, ... , 1, 0]$, and this path has length at most $\ceil{\frac{k}{3}}$ (see Figure \ref{AB_3}). Once $s$ is informed, broadcasting in the rest of the graph can be completed in at most $k$ additional time units using Broadcast Scheme \ref{binoscheme}.
\end{proof}

\begin{theorem}
    For any $k$, $n$ such that $2^{k-1} + 1 \leq n \leq 2^{\floor{\frac{3k}{4}}-1}$, there exists a planar broadcast graph on $n$ vertices with broadcast time $k$.
\end{theorem}

\begin{proof}
    To obtain such a graph, modify the construction of an $EB_k$ in Section \ref{SP_construction} by letting $G_1 = AB_{k-1}$ and $G_2 = AB_{\floor{\frac{3k}{4}-1}}$. The graph contains $n = 2^{k-1} + 2^{\floor{\frac{3k}{4}-1}}$ vertices with $\ceil{\log n} = k$, and is planar by being defined as the parallel composition of two planar graphs (see Lemma \ref{abk_planar_lemma}). 

    We show that the broadcast time of this graph is equal to $k$ by providing a broadcast scheme for any vertex. We use the broadcast scheme for any originator in an $EB_k$ (Theorem \ref{thrm_ebk}) with only one modification: $v$ broadcasts in its $BT_{\floor{\frac{3k}{4}-1}}$ subtree instead of its $BT_{\floor{\frac{k}{2}}}$ subtree.

    Recall that after time 1, vertices $t$ and $v$ are informed. Again, $v$ is part of some $BT_{\floor{\frac{3k}{4}-1}}$ subtree (with added shortcut edges), which is not necessarily $G_2 \setminus \{t\}$, and which $v$ begins informing at time 2. It is then sufficient to show that all vertices in this subtree can be informed in $k-1$ additional time units to complete broadcasting by time $k$. By Lemma \ref{btimeABk}, vertex $v$ can inform the $BT_{\floor{\frac{3k}{4}-1}}$ subtree it is contained in within $\left\lceil{\frac{4 (\floor{ \frac{3k}{4}} -1)  }{3}}\right\rceil \leq \left\lceil{\frac{  3k - 4  }{3}}\right\rceil = k-1$ additional time units.
    
    For $n$ with $2^{k-1} < n < 2^{k-1} + 2^{\floor{\frac{3k}{4}-1}}$ we can again remove up to $2^{\floor{\frac{3k}{4}-1}}-1$ vertices from $G_2 \setminus \{s, t\}$ starting with the leaves without increasing the broadcast time.
\end{proof}

\section{Future Work}

Having presented bounds on the broadcast times of subclasses of SP graphs and having constructed SP and planar graphs with efficient broadcast times, some questions remain to be studied.

\begin{itemize}
    \item What are other subclasses of SP graphs with bounded broadcast time?
    \item Are there other SP graphs on $n$ vertices with logarithmic broadcast time and maximum degree at most $\ceil{\log n}$?
    \item Are there SP broadcast graphs on roughly $2^{k-1} + 2^{\floor{\frac{k}{2}}}$ vertices without a universal vertex, for any $k$?
    \item What is the maximum number of vertices in an SP broadcast graph?
    \item The question in \citep{HellPaper} remains open: do planar broadcast graphs on a large number of vertices require a universal vertex?
    \item Are there planar broadcast graphs on more than $2^{k-1} + 2^{\floor{\frac{3k}{4}}-1}$ vertices, for any $k$?
\end{itemize}

\bibliographystyle{plain}
\bibliography{bibliography}

\end{document}